\documentclass[a4paper, 11pt]{nyjm}

\usepackage{amsmath,latexsym,amssymb,amscd}
\usepackage[utf8]{inputenc}
\usepackage[all]{xy}
\usepackage[british]{babel}
\usepackage{url}
\usepackage{enumerate}

\usepackage{enumitem}

\title{A sharp Bogomolov-type bound}

\author[{S. Checcoli}]{Sara Checcoli}
\address{Mathematisches Institut\\
Universität Basel\\
Rheinsprung 21\\
CH-4051 Basel\\
Switzerland}
\email{sara.checcoli@unibas.ch}
\author[{F. Veneziano}]{Francesco Veneziano}
\address{Mathematisches Institut\\
Georg-August Universität Göttingen\\
Bunsenstraße 3-5\\
D-37073 Göttingen\\
Deutschland}
\email{fvenez@uni-math.gwdg.de}
\author[{E. Viada }]{Evelina Viada}
\address{Mathematisches Institut\\
Universität Basel\\
Rheinsprung 21\\
CH-4051 Basel\\
Switzerland}
\email{evelina.viada@unibas.ch}

\thanks{All authors are supported by the Swiss National Science Foundation (SNSF)}

\keywords{Abelian varieties, Essential minimum, Bogomolov}

\subjclass{11G10, 11G50, 11J95}

\renewcommand{\epsilon}{\varepsilon}

\newcommand{\idi}{\mathrm{Id}}
\newcommand{\defalf}{\alpha}
\newcommand{\mat}{{\mathrm{Mat}}}
\newcommand{\moltp}{\binom{N-1}{n-1}}

\newcommand{\abv}{2}
\newcommand{\cnbundle}{{\mathcal{L}_{N}}}
\newcommand{\cbo}{{\mathcal{L}_{n}}}

\newcommand{\chern}{{\rm{c_1}}}

\newcommand{\invf}{{\hat\phi}}
\newcommand{\piuc}{\boxplus}

\newcommand{\ef}{\phi}
\newcommand{\efdual}{\hat{\phi}}

\newcommand{\elle}{{\mathcal{L}}}

\newcommand{\Ii}{\mathbf{I}}
\newcommand{\Mm}{\mathcal{M}}

\newcommand{\qe}{\mathbb{Q}}

\newcommand{\pu}{\mathbb{P}}

\DeclareMathOperator{\codim}{codim}

\DeclareMathOperator{\Hom}{Hom}
\DeclareMathOperator{\stab}{Stab}
\newcommand{\stabv}{\stab Y}

\newtheorem{thm}{Theorem}[section]

\newtheorem{propo}[thm]{Proposition}
\newtheorem{lem}[thm]{Lemma}

\newtheorem{D}[thm]{Definition}

\begin{document}

\begin{abstract}
We prove a sharp lower bound for the essential minimum of a non-translate variety in certain abelian varieties.
This uses and generalises a result of Galateau. Our bound is  a new step in direction of an abelian analogue by David and Philippon of a toric conjecture of Amoroso and David and has applications in the framework of anomalous intersections.
\end{abstract}
\maketitle
\tableofcontents
\section{Introduction}
In this paper, by variety we mean an algebraic variety defined over the algebraic numbers.
Let $A$ be an abelian variety; with a symmetric ample line bundle $\elle$ on $A$ we associate an embedding $i_\elle: A \to \pu_m$ defined by the minimal power of $\elle$ which is very ample. Heights and degrees corresponding
to $\elle$ are computed via this embedding. More precisely, $\hat{h}_\elle$ will denote the
$\elle$-canonical   Néron-Tate height, and the degree $\deg_\elle$ of a subvariety of $A$ is defined as the degree of its image in $\pu_m$ under $i_\elle$.

If $A$ is a product of abelian varieties, we fix on each simple non-isogenous factor a symmetric very ample line bundle. On $A$ we consider the line bundle $\elle$ obtained as the tensor product of the pullbacks of the natural projections of $A$ onto its factors.

A subvariety $Y$ of an abelian variety is \emph{translate} if
it is the union of translates of algebraic subgroups.
We also say that a subvariety $Y\subseteq A$ is \emph{torsion} if it is the union of components of algebraic subgroups. An irreducible $Y$ is called \emph{transverse} (resp. \emph{weak-transverse}) if it is
not contained in any proper translate (resp. in any proper torsion variety).

The torsion is a dense subset of $A$ and, in view of the abelian analogue of Kronecker's theorem, it is exactly the set of points of height 0. This statement motivates further questions about points of small height on $A$ and on its subvarieties.

In the context of the Lehmer problem, effective lower bounds for the height of non-torsion points of $A$ have been studied for example in \cite{SilvermanLowerBoundsDuke}, \cite{MasserSmallValues87}, \cite{davidhindry} and \cite{BakSil}.

One is then led to study the height function on an algebraic subvariety $Y$ of $A$. For instance, by the Manin-Mumford conjecture (proved by Raynaud \cite{RaynaudMM}, \cite{RaynaudMMgen}), the points of height zero are dense in $Y$ if and only if $Y$ is torsion.

More in general, setting \[Y(\theta)=\{x \in Y(\overline\qe) \mid \hat{h}_\elle(x) \leq \theta\}\] and denoting by $\overline{Y(\theta)}$ its Zariski closure, we have the following result, known as the Bogomolov Conjecture, proved by Ullmo and Zhang (see for instance \cite{UllmoBogo}  and \cite{ZhangEquidistribution}).

\begin{thm}[Bogomolov Conjecture] The essential minimum 
$$\mu_\elle(Y)=\sup\{\theta \mid \overline{Y(\theta)}\subsetneq Y\}$$ 
is strictly positive if and only if $Y$ is non-torsion.
\end{thm}

\bigskip

The problem of giving explicit bounds for $\mu_\elle(Y)$ in terms of geometrical invariants of $Y$ and of the ambient variety has been studied in several deep works, for instance in \cite{DavidPhilipponTiruchirapalli}, \cite{sipacommentari}, \cite{sipaIMRP} in the abelian case and in \cite{BZ}, \cite{SchmidtHeightsPoints}, \cite{fra}, \cite{iofrancesco} in the corresponding toric case.

An effective Bogomolov Conjecture (up to $\eta$) for transverse varieties in an abelian variety $A$  states the following.  

{\it Effective Bogomolov Bound } : For any abelian variety and for any real $\eta>0$ there exists a positive effective constant $C(A,\elle,\eta)$ such that for every transverse subvariety $Y$ of $A$
\begin{equation}\label{galkind}\mu_{\elle}(Y)\geq {C(A,\elle, \eta)}\frac{1}{\left({\deg_{\elle}Y}\right)^{\frac{1}{\codim Y}+\eta}}.\end{equation}

Galateau, in \cite{galateau}, proves this   lower bound
for  varieties transverse in an  abelian variety with  a positive density of ordinary primes (Hypothesis H in \cite{galateau}). According to a conjecture of Serre, this shall always be true (see \cite[\S 7]{PinkConjDensity}). Some cases are  proved: CM abelian varieties (see \cite[part~5]{BakSil}), powers of elliptic curves (see \cite[chap.~IV]{SerreDensityEllCurv}) and abelian surfaces (see \cite[part~VI, 2.7]{OgusDensitySurf}) have a positive density of ordinary primes. Thus the bound \eqref{galkind}  holds  in all  such abelian varieties. The constant of Galateau is  effective, for instance, for  powers of elliptic curves with the canonical embedding.

In this paper we prove a strong result for non-translates, in the following sense.
For an irreducible variety
$Y$, we can consider the
minimal translate $H$ which contains $Y$.
 It is then natural to give the following definition.
\begin{D}
An irreducible variety $Y$  is \emph{transverse  in a translate $H$} if $Y\subseteq H$  is not contained in any translate strictly contained in $H$.
The relative codimension of $Y$ is then  the codimension of $Y$ in $H$ denoted $\codim_H Y$.
\end{D}

In our main theorem, we prove that an effective Bogomolov  bound like in \eqref{galkind} implies an $H$-effective Bogomolov   bound,  in which $\deg H$ works in our favour.

\begin{thm}
\label{treapp} Let $A$ be an abelian variety. Assume that an effective Bogomolov bound  \eqref{galkind} holds. Then, for every variety $Y$ transverse in a translate $H\subset A$, 
$$\mu_\mathcal{L} (Y) \ge c(A,\elle,\eta)  \frac{
( \deg_\mathcal{L} H)^{\frac{1}{\codim_H Y}- \eta}
}{
(\deg_\mathcal{L} Y)^{\frac{1}{\codim_H Y}+ \eta}
},$$
where  $c(A,\elle,\eta)$ is a positive effective constant.
 \end{thm}
 As described in Section \ref{1.3implica1.4}, from the result of Galateau we immediately deduce: 
 \begin{thm}
\label{dueapp} Let $A$ be an abelian variety with a positive density of ordinary primes. Let  $Y$ be a variety transverse in a translate $H\subset A$. Then, for any
$\eta>0$, there exists a positive   constant $c'(A,\elle,\eta)$ such that 
$$\mu_\mathcal{L} (Y) \ge c'(A,\elle,\eta)  \frac{
( \deg_\mathcal{L} H)^{\frac{1}{\codim_H Y}- \eta}
}{
(\deg_\mathcal{L} Y)^{\frac{1}{\codim_H Y}+ \eta}
}.$$
 \end{thm}

\bigskip

In our bound, as in the effective Bogomolov bound,  the degree of $Y$ appears in the denominator, but here the degree of $H$ appears at the numerator. This last dependence is crucial for applications in the context of torsion anomalous intersections and of the Zilber-Pink conjecture, enabling a significant simplification in some classical uses of such Bogomolov-type bounds. For instance, using a special case of Theorem \ref{dueapp}, we obtain some new results. In particular, in \cite{TAI}  we show that: If $V$ is a weak-transverse variety in a product of elliptic curves with CM, then the $V$-torsion anomalous varieties of relative codimension one are finitely many.  In addition, their degree and normalised height are effectively bounded; see \cite[Theorem~1.3]{TAI} for a precise statement.

\bigskip

Our theorems are analogues, up to the logarithmic correction factor, of a toric conjecture of Amoroso and David in \cite{fra}; see \cite{sipaIMRP} for a suitable conjecture in the abelian case.  In the context of the Lehmer problem, a lower bound of similar nature to  ours for translates of tori is given by  Sombra and Philippon in \cite{SombraPatrice} and 
for points in CM abelian varieties is given by Carrizosa in \cite{carrizosaIMRN}. 

\bigskip

The main point in our proof is to compare the line bundle $\elle$ on $A$ with its restriction $\elle|_H$ to $H$.
To overcome this difficulty we prove an equivalence of line bundles which enables us to describe the restriction of the line bundle to $H$ in terms of pull-backs  through several isogenies.
From the bound \eqref{galkind}, we immediately deduce the good behaviour of essential minima under isogenies.
We notice that our constant depends explicitly on the constant in \eqref{galkind}. In addition, our method applies to any lower bound of the kind \eqref{galkind} where the hypothesis on $Y$  are preserved by isogenies.
 We finally remark that the hypothesis of irreducibility for $Y$ is not restrictive, indeed the essential minimum of a reducible variety is the maximum of the essential minima of its components.

\medskip

This paper is structured in the following way.
In Section \ref{prelimapp}, we give some preliminaries. Then, as a straightforward consequence of bound \eqref{galkind}, we show the good behaviour of the essential minimum under the action of an isogeny. 

In Section \ref{equivalencesection} we present the main technical ingredients of the proof, which is an equivalence of line bundles.

In Section \ref{sei} we finally prove our main theorem in two steps:
we first consider the case when $H$ is an abelian subvariety, and then we reduce to this case the general case of $H$ a translate.

\section{Preliminaries}\label{prelimapp}
Let $A$ be an abelian variety.
We consider irreducible subvarieties  $Y$ of $A$ and we shall analyse the sets of points of small height on $Y$ with respect to different line bundles; for this reason, when talking about heights, degrees and essential minima, we  always indicate which line bundle we are using.

An abelian variety is isogenous to a product of simple abelian varieties.  We suppose 
\begin{equation}\label{decompA}
A=A_1^{N_1}\times\dotsb \times A_r^{N_r},
\end{equation} 
where the $A_i$'s are pairwise non-isogenous simple abelian varieties, and we
fix embeddings $A_i \hookrightarrow \mathbb{P}_{m_i}$, given by symmetric ample line bundles on $A_i$. On any product variety we consider the line bundle $\elle$ obtained as the tensor product of the pullbacks of the natural projections on its factors.

If $A$ is as in \eqref{decompA} and $B$ is an abelian subvariety  of $A$, then \[B=B_1\times\dotsb \times B_r,\] where for every $i$, $B_i$ is an abelian subvariety of $A_i^{N_i}$; this follows, for instance, from \cite[Lemma 7, p.~262]{MWZeroEst}  because $A_i^{N_i}$ and $A_j^{N_j}$ do not have any non-trivial isomorphic subquotients, when $i\neq j$.

 We denote by $\mathrm{End}(-)$ the ring of endomorphisms of abelian varieties. 
Recall that $\mathrm{End}(A_i)\otimes_{\mathbb{Z}}\mathbb{R}$ is a real, complex or quaternionic algebra, and in each of these cases we have a standard euclidean norm.

We first prove several results which hold for a power of a simple abelian variety. Then we extend them to a general abelian variety.

In the following, we write $\ll$ and $\gg$ to indicate inequalities up to constants depending only on the ambient variety, the fixed line bundles and a positive real number $\eta$, but not on the variety $Y$ or the isogenies involved, which may vary.

\subsection{Basic relations of degrees}
We  recall some basic properties of degrees of subvarieties. In this section, we assume that $A$ is an abelian variety of dimension $D$.

\medskip

Let  $m$ be a positive integer. Let $\elle$ be a symmetric ample line bundle on $A$ and let $\elle^{m}$ be the tensor product of  $m$ copies of $\elle$. Let  $Y$ be an irreducible algebraic subvariety of $A$ of dimension $d$. Then,
\begin{equation}
\label{gradino1ii}\deg_{\elle^{m}}Y=m^{ d}\deg_\elle Y.\end{equation}

\medskip

We are interested in how the degree changes under the action of the multiplication morphism.
For  $a \in \mathbb{Z}$, we denote by $|a|$ its absolute value and by $[a]$ the multiplication by $a$ on $A$. 

From \cite[Corollary~3, p.~59]{MumfAV}, we have
\begin{equation*}
[a]^{*}\elle=\elle^{a^2}.
\end{equation*}
Hindry \cite[Lemma 6]{Hin} proves
\begin{equation*}
\label{hininv}\deg_\elle [a]^{-1}Y=|a|^{2(D-d)}\deg_\elle  Y
\end{equation*}
and
\begin{equation}
\label{hinm} \deg_\elle  [a]Y=\frac{|a|^{\abv d}}{|\stabv\cap \ker[a]|}\deg_\elle Y.\end{equation}

Let $\phi:A \to A$ be an isogeny; then the Projection Formula gives
\begin{equation}\label{projectionformula}\deg_{\ef^*\elle}Y=\deg_\elle \ef_*(Y),
\end{equation}
where $\ef_*(Y)$ is the cycle  with support $\ef(Y)$ and multiplicity $\deg (\ef_{|Y})$.

Furthermore, by \cite[Corollary 3.6.6]{l-b} we have
 $$\deg_{\phi^*\mathcal{L}}A=|\ker \phi |\deg_{\mathcal{L}}A,$$
and more in general: 
 \begin{lem}
\label{gradino}
 Let $\phi:A \to A $ be an isogeny.
 Let $Y$ be an irreducible algebraic subvariety of $A$. Then
 $$\deg_\elle \phi_*(Y)=|\stabv\cap \ker \phi|\deg_\elle \phi(Y).$$
\end{lem}
We also have: \begin{lem}
\label{ordinestab}
Let $\phi:A \to A $ be an isogeny. Then,
\begin{enumerate}[label=\roman{*}., ref=(\roman{*})]
\item\label{ordinestabi} $\stab \phi^{-1}(Y)=\phi^{-1}(\stabv)$.
\item\label{ordinestabii} Let $\hat\phi$ be an isogeny such that  $\hat\phi \phi=\phi \hat\phi=[a]$. Then  $$|\stab \hat\phi^{-1} (Y) \cap \ker[a]|=|\ker\hat\phi| |\stabv \cap \ker \phi|.$$
\end{enumerate}
\end{lem}
\begin{proof}
Part \ref{ordinestabi}: Let $t\in \stab \phi^{-1}(Y)$ then $$ \phi^{-1}(Y)+t = \phi^{-1}(Y).$$ Taking the image, $Y+\phi(t) = Y$ and $\phi(t) \in \stab Y$. That gives $t \in \phi^{-1}(\stabv)$. Conversely, let $t \in \stabv$, then $$Y+t= Y$$ and taking the preimage $\phi^{-1}(Y+t) = \phi^{-1}(Y)$. Then $\phi^{-1}(t) \in \stab \phi^{-1}(Y)$.

Part \ref{ordinestabii}: By part \ref{ordinestabi} applied to $\hat\phi$, we have $\stab \hat\phi^{-1} (Y)=\hat\phi^{-1} (\stabv)$. As $\phi \hat\phi=[a]$, $\ker[a]=\hat\phi^{-1}(\ker\phi)$.
Then
\[\stab \hat\phi^{-1} (Y) \cap \ker[a]=\hat\phi^{-1}\left(\stabv \cap \ker \phi \right).\qedhere\]
\end{proof}

\subsection{Basic relations of  essential minima}
We now investigate useful relations for the essential minimum.

Recall that, by definition, for every $x\in A$ and isogeny $\phi:A \to A$, we have $\hat h_{\ef^*\elle}(x) = \hat h_\elle(\phi(x))$; then
\begin{equation}\label{phisaltamu}
\mu_{\phi^*\mathcal{L}}(Y)=\mu_{\mathcal{L}}(\phi(Y)).\end{equation}
In addition \begin{equation}
\label{gradino1iii}\mu_{\elle^{m}} (Y)=m \mu_\elle(Y).\end{equation}
Another easy remark is stated in the following lemma.

\begin{lem}
  \label{minessmin} 
  Let $\elle_1,\elle_2$ be two ample line bundles on $A$. Then for every irreducible subvariety $Y\subseteq A$, 
  $$\mu_{\elle_1 \otimes \elle_2}(Y) \ge \mu_{\elle_1 }(Y)+ \mu_{\elle_2 }(Y).$$
\end{lem}
  
\begin{proof}
 This lemma is proved by contradiction, and it relies on the height relation $\hat h_{\elle_1 \otimes \elle_2} (x) = \hat h_{\elle_1 } (x)+\hat h_{\elle_2 } (x) $ for every $x\in A$.
   
 Suppose, by contradiction, that  
 $ \mu_{\elle_1 \otimes \elle_2}(Y) < \mu_{\elle_1 }(Y)+ \mu_{\elle_2 }(Y).$   Then there exist reals $k_1,k_2$ such that $0<k_i<\mu_{  \elle_i} (Y) $ and $\mu_{\elle_1 \otimes \elle_2}(Y)<k_1+k_2$, and a dense subset $U$ of $Y$ such that
 \begin{equation}
   \label{stella}  \hat h_{\elle_1 \otimes \elle_2}(x)\le k_1+ k_2 \quad \forall x\in U.
 \end{equation} 

From the definition of $\mu_{ \elle_i} (Y) $, the set of points of $Y$ such that $\hat h_{\elle_i }(x) \leq k_i$ is contained in a closed subset $V_i\subsetneq Y$. Since $U$ is dense, $U'=U\setminus{\cup_i V_i}$ is also dense in $Y$. In addition, for every $x\in U'$, $ \hat h_{ \elle_i}(x) > k_i $.  Then $$\hat h_{\elle_1 \otimes \elle_2} (x) =  \hat h_{ \elle_1} (x)+\hat h_{ \elle_1} (x) >k_1+k_2 \quad \forall x\in U'$$   which contradicts \eqref{stella}.
\end{proof}

\subsection{Theorem \ref{treapp} implies Theorem \ref{dueapp} }\label{1.3implica1.4}

The implication `Theorem \ref{treapp} implies Theorem \ref{dueapp}'  is straight forward from a result of Galateau  \cite{galateau}. For convenience we recall his theorem.   

For $Y$ a subvariety of a (semi)abelian variety, with an ample line bundle $\elle$, define
\[
\omega_\elle(Y)=\min_Z\{\deg_\elle Z\}
\]
where the minimum is taken over all the hypersurfaces (not necessarily irreducible) containing $Y$. 
\begin{thm}[\cite{galateau}, Theorem 1.1]\label{galateaucomm}
Let $B$ be an abelian variety  with a positive density of ordinary primes, and $\mathcal{L}$ be an ample and symmetric line bundle on $B$. Let $Y\subseteq B$ be a transverse variety.
Then 
\[
\mu_{\elle}(Y)\geq \frac{{C_0}(B,\elle)}{\omega_\elle(Y)}\left(\log\left(3\deg_\elle Y\right)\right)^{-\lambda(Y)}
\]
where ${C_0}(B,\elle)$ is a positive real depending on the variety $B$ and the line bundle $\elle$, and $\lambda(Y)=(5\dim B(1+\codim_B Y))^{1+\codim_B Y}$.
In particular, for every $\eta>0$, there exists a constant $C(B,\elle,\eta)$ such that
\begin{equation*}
\mu_{\elle}(Y)\geq \frac{C(B, \elle,\eta)}{(\deg_\elle Y)^{\frac{1}{\codim_B Y}+\eta}}.
\end{equation*}
\end{thm}
Note that if $A$ has a positive density of ordinary primes, then also all its abelian subvarieties have it.

 As mentioned in the introduction the hypothesis on a positive density of ordinary primes always holds for CM abelian varieties, powers of elliptic curves  and  abelian surfaces and  is conjectured to hold for every abelian variety.

\subsection{Essential minimum and isogenies}\label{isogenyproof}
In this section we  derive from the effective Bogomolov bound \eqref{galkind} a theorem which shows the good behaviour of the essential minimum under the action of an isogeny. 
 This result will be among the ingredients used, in Section \ref{sei}, to prove the main Theorem. 

\begin{thm}\label{thm1}
Let  $B$ be abelian variety such that an effective Bogomolov bound  \eqref{galkind} holds.
Let $Y$ be a transverse subvariety of $B$ and $\phi:B \to B$ an isogeny.
Then for every $\eta>0$ we have
  $$\mu_{\phi^*\elle}(Y)\gg \frac{\left({\deg_{\phi^*\elle}B}\right)^{\frac{1}{\codim_B Y} -\eta}}{\left({\deg_{\phi^*\elle}Y}\right)^{\frac{1}{\codim_B Y} +\eta}}.$$
\end{thm}

\begin{proof}
 Let $g=\dim B$, and let  $a$ be an integer of minimal absolute value 
such that there exists an isogeny  $\efdual$  with $\ef \efdual =\efdual \ef=[a]$.  By definition of dual we have $|a| \leq \deg \phi$.

Let  $W$ be  an irreducible component of ${\efdual}^{-1}(Y).$
Note that isogenies preserve dimensions and transversality, so $\dim W=\dim Y=d$ and $W$ is transverse.

We have 
$$[a] W=\ef \efdual W= \ef \efdual{\efdual}^{-1}(Y) =\ef(Y).$$
Then \begin{equation}
\label{miniesi}\mu_{\ef^*\elle}(Y)=\mu_\elle(\phi(Y))=\mu_\elle([a]W)=|a|^\abv  \mu_\elle(W).
\end{equation}
We  now need to estimate $\deg_\elle W$.

Since $Y$ is transverse, $d>0$; then by formula \eqref{hinm}, 
$$\deg_\elle \ef (Y)=\deg_\elle [a]W=\frac{|a|^{\abv d}}{|\stab W \cap \ker[a]|}\deg_\elle W$$
or equivalently
$$\deg_\elle W=\frac{|\stab W \cap \ker[a]|}{|a|^{\abv d}}\deg_\elle \ef (Y).$$
 Using Lemma \ref{ordinestab} \ref{ordinestabii} and Lemma \ref{gradino}, we obtain
\begin{equation*}
\begin{split}
\deg_\elle W &=\frac{|\ker \efdual|}{|a|^{2d}}  |\stab Y \cap \ker \phi| \deg_\elle \ef (Y)\\  
&=\frac {|\ker \efdual| }{|a|^{2d}} \deg_\elle \ef_* (Y).
\end{split}
\end{equation*}

Since $W$ is transverse, applying bound \eqref{galkind} we have
\begin{align}
\mu_\elle(W) & \gg  \left({\deg_\elle W}\right)^{-\frac{1}{g-d}-\eta} \label{dacambiare}\\
                        &=  \left(\frac{{|a|^{\abv d}} } {|\ker \efdual| \deg_\elle \ef_* (Y)  } \right)^{{\frac{1}{g-d}+\eta}} \nonumber\\
                        &=  \left(\frac{{|a|^{\abv d}} \deg_\elle B} {|\ker \efdual| \deg_\elle \ef_* (Y)  } \right)^{{\frac{1}{g-d}+\eta}}(\deg_\elle B)^{\frac{-1}{g-d}-\eta}  \label{degAconst}\\
                       &\gg \left(\frac{{|a|^{\abv d}} \deg_\elle B} {|\ker \efdual| \deg_\elle \ef_* (Y)  } \right)^{{\frac{1}{g-d}+\eta}}(\deg_\elle B)^{-\eta}, \nonumber
\end{align}
where line \eqref{degAconst} follows absorbing the appropriate power of $\deg_\elle B$ in the implicit constant.

We  substitute this last estimate in \eqref{miniesi}, then
\begin{equation*}
\begin{split} \mu_{\ef^*\elle}(Y)&=|a|^\abv  \mu_\elle(W)\\&\gg |a|^\abv \left(\frac{{|a|^{\abv d}} \deg_\elle B} {|\ker \efdual| \deg_\elle \ef_* (Y)  } \right)^{{\frac{1}{g-d}+\eta}}(\deg_\elle B)^{-\eta}\\
&=\left(\frac{|a|^{\abv g-\abv d} |a|^{\abv d}|\ker \phi| \deg_\elle B} {|\ker \efdual| |\ker \phi|\deg_\elle \ef_* (Y)  } \right)^{\frac{1}{g-d}+\eta} (\deg_\elle B)^{-\eta}|a|^{-\abv  (g-d) \eta}\\
&=\left(\frac{|a|^{\abv g}\deg_{\ef^*\elle}B} {|a|^{\abv g}\deg_\elle \ef_* (Y)  } \right)^{\frac{1}{g-d}+\eta} (\deg_\elle B)^{-\eta}|a|^{-\abv  (g-d) \eta}
\\&=\left(\frac{\deg_{\ef^*\elle}B} {\deg_{\ef^*\elle}  Y  } \right)^{\frac{1}{g-d}+\eta}(\deg_\elle B)^{-\eta}|a|^{-\abv  (g-d) \eta},
\end{split}
\end{equation*}
where  $|\ker \efdual| |\ker \phi|=|a|^{2g}$ because $\phi\hat\phi=\hat\phi\phi=[a]$. In addition $$\deg_{\ef^*\elle}B=|\ker \phi|\deg_{\elle}B=\deg
\phi  \deg_\elle B .$$ 
 Since $|a| \leq \deg \phi$,  we have  $$(\deg_\elle B)^{-\eta}|a|^{-2 (g-d) \eta}\ge
(|a|\deg_\elle B)^{-2 (g-d)
  \eta}\ge(\deg_{\ef^*\elle}B)^{-  2(g-d) \eta}.$$
Then 
$$\mu_{\ef^*\elle}(Y)\gg \frac{\left(\deg_{\ef^*\elle}B\right)^{\frac{1}{g-d}-  2(g-d) \eta+\eta}} {\left(\deg_{\ef^*\elle}  (Y) \right)^{\frac{1}{g-d}+\eta } },
$$
which  easily implies the wished bound, after changing $\eta$.
\end{proof}

\subsection{Morphisms}\label{subsecmorphapp}
Let $A$ be a simple abelian variety of dimension $D$ and let $M\le N$ be positive integers.
 
We can associate  a matrix $\mathcal{A}=(\psi_{ij})_{ij} \in \mat_{M\times N}(\mathrm{End}(A))$ with a morphism
\begin{align*}
\psi_\mathcal{A}:A^N &\to A^M\\
(x_1,\dots, x_N) &\mapsto ({\psi_{11}}x_1+\cdots +{\psi_{1N}}x_N,\dots,{\psi_{M1}}x_1+\cdots +{\psi_{MN}}x_N).
\end{align*}

We also associate a morphism $\psi:A^N \to A^M$ with a matrix
\[
\mathcal{A}_\psi =
\left( {\begin{array}{ccc}
 \psi_{11} & \ldots & \psi_{1N}\\
 \vdots & \ddots &\vdots\\
 \psi_{M1} & \ldots & \psi_{MN}
 \end{array} } \right)
\]
such that all components of $\ker \psi$ are components of
\begin{equation}\label{sistema_app}
\begin{cases}{\psi_{11}}x_1+\cdots +{\psi_{1N}}x_N=0\\
\vdots\\
{\psi_{M1}}x_1+\cdots +{\psi_{MN}}x_N=0\end{cases}
\end{equation}
and such that the product of the norms of the rows of $\mathcal{A}_\psi$ is minimised.  As a norm for the rows we take the euclidean norm in $(\mathrm{End}(A)\otimes_{\mathbb{Z}}\mathbb{R})^N$.

For notation's sake, sometimes we  identify the morphism with the matrix.

\bigskip

We associate a morphism $\psi:A^N \to A^M$ with an abelian subvariety $A'$ of $A^N$ given by the connected  component of $\ker\psi$  passing through the zero of $A^N$.

Finally, we associate an abelian subvariety  $A'\subseteq A^N$  of codimension $MD$ with  the projection morphism $$\psi:A^N \to A^M=A^N/A'.$$ Then $A'=\ker\psi$ is a component of  the variety defined by the system \eqref{sistema_app}, where $\mathcal{A}_{\psi}=(\psi_{ij})_{ij}$ is the matrix associated with $\psi$, as above. 

 We  call the morphism (resp. the matrix) just defined the \emph{associated morphism} (resp. \emph{associated matrix}) of $A'$.

We finally define the norm of a morphism in the usual way.

\section{An equivalence of line bundles}\label{equivalencesection}
In this section we let $A$ be a simple abelian variety of dimension $D$.
\subsection{Definitions}
Let now $B$ be
an abelian subvariety of $A^N$ of dimension $nD$.
\begin{lem}\label{mw}

Let $B$ be an abelian subvariety of $A^N$ of dimension $nD$. Then there exists an isogeny $\phi: A^N\to A^N$ defined by a matrix
 \begin{equation*}
 \mathcal{A}_\phi={ \binom{\phi_B}{\phi_{B'}}}\end{equation*}
such that:
\begin{enumerate}[label=\roman{*}., ref=(\roman{*})]
\item\label{mwi} $ \phi_B\in\mat_{(N-n)\times N}(\mathrm{End}(A))$ and  $ \phi_{B'}\in\mat_{n\times N}(\mathrm{End}(A))$ are both of full rank;
\item\label{mwii} $ \phi(B)=\{0\}^{N-n}\times A^n$;
\item\label{mwiii} $\deg  \phi\le c(A,N)$, where $c(A,N)$ is a constant depending only on $A$ and $N$. 
\end{enumerate}
\end{lem}

\begin{proof}
By a lemma of  Bertrand, (see \cite[Proposition~2, p.~15]{bertrand}) we can find a supplement $B'$ of $B$ such that $B+B'=A^N$ and  the cardinality  of $B\cap B' $ is bounded by a constant depending on $A$ and $N$. 

Then, there exists  a matrix  $$ \phi_B\in \mat_{{(N-n)}\times
  N}(\mathrm{End}(A))$$   of rank $N-n$ such that  $\ker  \phi_B= B+\tau$ for $\tau$ a torsion group contained in $B'$ of cardinality bounded by a constant depending only on $N$. 
Similarly, let  $ \phi_{B'}\in \mat_{{n}\times
  N}(\mathrm{End}(A))$   be a matrix of rank $n$ such that $\ker  \phi_{B'}= {B'}+\tau'$ for $\tau'$ a torsion group contained in $B$ of cardinality bounded only in terms of $N$ (see also Masser and W\"ustholz \cite[Lemma~1.3]{Masserwustholz}).

If we define the isogeny $\phi$ as follows:
 $$ \phi={ \binom{\phi_B}{\phi_{B'}}}: A^N\to A^N,$$
 then properties \ref{mwi} and \ref{mwii} are satisfied by construction, and we have that
\[|\ker  \phi|= |(B+\tau) \cap( B'+\tau')|=|(B\cap B')+\tau+\tau'|\ll 1.\qedhere\]
\end{proof}

It is an exercise in linear algebra to show that  there exists an isomorphism $T$ of norm bounded only in terms of $N$, such that all $n\times n$ minors of the matrix consisting of the last $n$ columns of $\left(\phi T\right)^{-1}$ have determinant different from zero.

 Clearly, for a point $P$, we have 
\begin{equation}\label{hupper}
\hat h(T(P))\ll ||T||^2 h(P).
\end{equation} Moreover, as in \cite[Proposition~4.2]{ioirmn} for elliptic curves and in \cite[Lemma~4.5]{ijnt} for abelian varieties, for any subvariety $X$ we get 
\begin{equation}\label{degupper}
\deg T(X)\ll ||T||^{2\dim X}\deg X.
\end{equation}
We deduce the following lemma.
\begin{lem}
 Let $T$, $P$ and $X$ be as above. Then
\begin{enumerate}[label=\roman{*}., ref=(\roman{*})]
\item\label{hb} $||T||^{-2(N-1)}\hat h(P)\ll \hat h(T(P)) \ll ||T||^2 \hat h(P)$
\item\label{degb} $||T||^{-2N} \deg X\ll \deg T (X) \ll  ||T||^{2\dim X}\deg X$. 
\end{enumerate}
\end{lem}
\begin{proof}
The upper bounds are immediate from \eqref{hupper} and \eqref{degupper}, while the lower bounds are obtained applying  \eqref{hupper}  to  $\hat T T(P)$ and and \eqref{degupper} to  $\hat T T(X)$. 
In particular, using \eqref{degupper},  one has
\[\deg (\hat T T(X))\leq ||\hat T||^{2 \dim T(X)}\deg T(X).\]
On the other hand
\[\deg (\hat T T(X))=\deg([\deg T] X);\]
which, combined with \cite[Lemma 6]{Hin}, give the desired bound.
\end{proof}
Hence $T$ changes degrees and heights by a constant depending only on $N$ and $D$.

Note that the isogeny $\phi:A^N \to A^N$ sends $B$ to the last $n$ factors, $$\phi({B})=0\times \dots \times 0 \times A^n.$$ 
We denote the immersion of $A^n$ in the last $n$ coordinates by  $$i: A^n \to A^N; \,\,\,(x_1,\dots,x_n)\mapsto (0,\dots,0,x_1,\dots,x_n).$$

\bigskip

Let us denote by $\defalf$ the minimal  positive integer  such that
there exists an isogeny $\hat\phi$  satisfying
$\phi\hat\phi=\hat\phi\phi=[\defalf]$.
 Note that, $|\defalf|\leq \deg \phi$.
Of course, we may take $\hat\phi$ to be the dual of $\phi$, but, eventually, for explicit computations, the above definition could be more convenient.

We decompose $\invf=(\mathcal{A}|\mathcal{B})$ with $\mathcal{A}\in \mat_{N \times (N-n)}(\mathrm{End}(A))$ and $\mathcal{B}\in \mat_{N\times n}(\mathrm{End}(A))$. We denote by $a_i$ the $i$-th row of $\mathcal{A}$; similarly
 $$\mathcal{B}=\left(\begin{array}{c}b_1\\
b_2\\
\vdots\\
b_N\end{array}\right)$$
and $(a_i,b_i)$ is the $i$-th row of $\hat\phi$.
\begin{lem}
\label{isogenie}
For $I \in \mathbb{I}=\{(i_1,\dots, i_n)\in \{1,\dots , N\}^{n}\text{ and } i_j<i_{j+1}\}$ the morphism 
 $$\varphi_I=\left(\begin{array}{c}b_{i_1}\\
\vdots\\
b_{i_n}\end{array}\right):A^n\to A^n$$ is an isogeny. 
\end{lem}
\begin{proof}
We are assuming that the last $n$ columns of $\phi^{-1}$ have $n\times n$ minors with non-zero determinant. Together with $\phi \invf=\defalf \idi_N$,  this implies that all $n\times n$ minors  of $\mathcal{B}$ have non-zero determinant.  Then $\det \varphi_I\neq 0$. This is equivalent to say that $\varphi_I$ is an isogeny, as $\varphi_I$ is a morphism associated with a square matrix of full rank.
\end{proof}
We remark that the map $\varphi_I:A^n\to A^n$ may be described as the composition
\begin{align*}
 A^n \stackrel{i}\hookrightarrow A^N \stackrel{\hat\phi}{\longrightarrow} A^N \stackrel{\pi_I}{\longrightarrow} A^n,
\end{align*}
where $\pi_I$ is the projection on the $n$ coordinates appearing in the multi-index $I$.

Indeed, we can sum up the situation with the following commutative diagram:
\[
\begin{xy}
   (0,20)*+{A^N}="v1";(20,20)*+{A^N}="v2";(40,20)*+{A^N}="v3";%
   (0,0)*+{B}="v4";(20,0)*+{A^n}="v5";(40,0)*+{A^n}="v6";%
   {\ar@{->}@/^{1pc}/|{[\alpha]}    "v1";"v3"};
   {\ar@{->}_{\phi} "v1";"v2"};{\ar@{->}_{\hat\phi} "v2";"v3"};%
   {\ar@{^{(}->}^{j} "v4";"v1"};{\ar@{^{(}->}^{i} "v5";"v2"};{\ar@{->}^{\pi_I} "v3";"v6"};%
   {\ar@{->}|{\phi_{|B}} "v4";"v2"};{\ar@{->}|{\phi_I} "v2";"v6"};%
   {\ar@{->}_{\varphi_{B}} "v4";"v5"};{\ar@{->}_{\varphi_I} "v5";"v6"};%
\end{xy}
\]

where $j$ is the inclusion of $B$ in $A^N$, $\phi_I$ is defined as the composition $\pi_I\circ\hat\phi$, and $\varphi_B$ is the restriction to $B$ of the  morphism associated with the matrix $\phi_{B'}$ in Lemma \ref{mw}. The fact that $\phi(B)=\{0\}^{N-n}\times A^n$ allows us to say that all horizontal arrows are isogenies.

\bigskip

With this notation, we remark a consequence of Lemmas \ref{gradino} and \ref{mw}. For $\phi$ as above we have
\begin{equation}\label{push}\deg_\elle  {\varphi_B}_*({B})=\deg \varphi_B \deg_\elle A^n \leq \deg\phi\deg_\elle A^n \ll \deg_\elle A^n.\end{equation}
In fact $\deg \varphi_B = |B\cap \ker \phi|\leq |\ker\phi|$.

Note that by our definitions of degree in $A^N$ and $A^n$, the map $i$ preserves the degree of subvarieties of $\{0\}^{N-n}\times A^n$.

\subsection{The equivalence}\label{cinque}
In this section $A$ is a simple abelian variety. We recall that on $A$ we fixed a symmetric ample line bundle $\elle_1$; for any integer $m$, we denote by $\elle_m$ the bundle on $A^m$ obtained as the tensor product of the pullbacks of the natural projections of $A^m$ on its simple factors. We now study
$\mathcal{L}_{N| B}$ for $B$ an abelian subvariety of $A^N$. We can express a power of $\mathcal{L}_{N| B}$ as  a tensor
product of pull-backs via different morphisms of the bundle
$\mathcal{L}_n$ on $A^n$. Of course, even if $B$ is isogenous to $A^n$, in general $\mathcal{L}_{N| B}$ is not the natural bundle $\elle_n$.

\bigskip

\begin{thm}\label{relchiave}
   The following equivalence of line bundles holds
   $$ {\mathcal{L}_{{N}_{\,\,\,{\big{|{B}}}}}^{ \defalf^2 \moltp }}\cong \phi_{|B}^* \bigotimes_I \phi^*_I \elle_n\cong\varphi_{B}^* \bigotimes_I \varphi^*_I \elle_n ,$$
   where $\phi\hat\phi=\hat\phi \phi=[\defalf]$ is as  above and $\varphi_I$ is defined in Lemma \ref{isogenie}. 
\end{thm}
\begin{proof}
We denote by $\chern(\elle)$ a representative of the first Chern-class of $\elle$. By $\piuc$ we mean the sum of cycles and by $\elle^m$ we mean  as before the tensor product of $m$ copies of $\elle$.
We recall that, for $f$ a morphism, we have 
\begin{equation}\label{chernmorfismo}
\chern(f^*\elle)=f^{*}\chern(\elle).
\end{equation}
   
Let $e_i:A^n \to A$ and $f_i:A^N \to A$ be the  projections on the $i$-th factor. Note that $e_i$'s (resp. the $f_i$'s) generate a free $\mathbb{Z}$-module of rank $n$ (resp. rank $N$) of $\Hom(A^n,A)$ (resp. $\Hom(A^n,A)$).
By definition of standard line bundle, 
\begin{align}
\label{chernoennepiccolo} \chern(\cbo)&=\piuc_{i=1}^n  e_i^*(\elle),\\
\nonumber\chern(\cnbundle)&=\piuc_{i=1}^N  f_i^*(\elle).
\end{align}
Note that, for any integer $\defalf$,  $(\defalf f_i)^* = (f_i[\defalf])^*=[\defalf]^{*} f_i^*$, and more in general, for $\psi\in \mat_{N\times m}(\mathrm{End}(A))$ and $\psi' \in \mat_{m'\times N}(\mathrm{End}(A))$ with $m,m'\in \mathbb{N}^*$, we have 
\begin{equation}\label{composizionepsi}
\psi^{*} (\psi')^*=(\psi'\psi)^*.
\end{equation}

In addition, by \cite[Corollary~3.6, p.~34]{l-b} for $\Mm$ a symmetric ample line bundle on $A^n$,  
\begin{equation*}
   \chern([\defalf]^*\Mm)=\chern(\Mm^{\defalf^2}),
\end{equation*}
and then, 
\begin{equation}\label{alpostodikercan}
  \chern(\mathcal{L}_N^{\defalf^2})=\piuc_{i=1}^N  (\defalf f_i)^*(\elle).
\end{equation}

Two line bundles are equivalent if and only if they have the same Chern-class;  the proof of the theorem is based on this remark.

      For any  morphism $\psi:A^N \to A^n$, applying \eqref{chernmorfismo}, \eqref{chernoennepiccolo} and \eqref{composizionepsi}, we have
      $$\chern(\psi^*\cbo)=\psi^{*} \chern(\cbo)=\piuc_{i=1}^n (e_i \psi)^*(\elle)=\piuc_{i=1}^n  \psi_i^*(\elle),$$ where $\psi_i: A^N \to A$ is  the $i$-th row of $\psi$.
   
   Apply this formula to each $\phi_I$. Then,  for $I=(i_1,\dots ,i_n)$,
   \begin{equation}\label{cherni}
      \chern(\phi_I^*\cbo )=  (a_{i_1},b_{i_1})^*(\elle) \piuc\dots \piuc (a_{i_n},b_{i_n})^*(\elle).
   \end{equation} 
   Since the Chern class of the tensor product is the sum of the Chern classes, we obtain
   \begin{multline}\label{chernl}
      \chern\left(\bigotimes_I \phi^*_I\cbo \right)= \boxplus_{I\in \mathbb{I}} \left( (a_{i_1},b_{i_1})^*(\elle)\piuc\dots \piuc (a_{i_n},b_{i_n})^*(\elle) \right) =\\
      =\binom{N-1}{n-1} \piuc_{i=1}^N  (a_i,b_i)^*(\elle),
   \end{multline}
   where the last equality  is justified from the fact that each multi-index $I$ consists of  $n$ coordinates and  each of the $N$ indices  appears the same number of times, so each row $(a_i,b_i)$ appears $\binom{N-1}{n-1}$ times, once for each possible  choice of the other $n-1$ rows.

   Recall that $\hat\phi\phi=[\defalf]$; by \eqref{chernmorfismo}  
   \begin{multline*}
      \chern\left(\phi^*\bigotimes_I \phi^*_I\cbo\right)=\phi^{*}\chern\left( \bigotimes_I \phi^*_I\cbo\right)=\\
      = \moltp \piuc _{i=1}^N ((a_i, b_i)  \phi)^*(\elle)
      =\moltp \piuc _{i=1}^N (\defalf f_i)^*(\elle).
   \end{multline*}

    Hence, using \eqref{alpostodikercan} and restricting to $B$ we conclude
    \begin{align*}
      \chern\left(\phi_{|B}^* \bigotimes_I \phi^*_I\cbo\right)&= \moltp \piuc _{i=1}^N {(\defalf f_i)^*(\elle)}_{|B} =\\
     =\moltp  \chern\left({\mathcal{L}^{\defalf^2}_N}_{|B}\right)&=\chern\left({\mathcal{L}_N^{ \defalf^2\moltp}}_{|B}\right).
    \end{align*}
The last isomorphism  in the statement of the theorem follows from $\phi_I\circ \phi_{|B}=\varphi_I \circ \varphi_B$.
\end{proof}

The following relation has been pointed out by Ga\"el R\'emond.
\begin{propo}\label{gael}
   The following equality holds
   $$\deg  \bigotimes_{I \in \mathbb{I}} \phi^*_I \cbo = \moltp^n \sum_{I \in \mathbb{I}} \deg  \phi_I^*\cbo.$$
\end{propo} 
\begin{proof}

 We  compute the degrees as intersection numbers. By relation
 \eqref{cherni}, we have
$$\deg \phi_I^* \cbo=n! \deg\prod_{i_j\in I} (a_{i_j}, b_{i_j})^* (\elle).$$
Similarly, by formula \eqref{chernl}, we obtain 
\begin{align*}
\deg {\bigotimes_I \phi_I^* \cbo}&=n! \moltp^n
\sum_{i_1<\dots <i_n}\deg \prod_{j=1}^n (a_{i_j}, b_{i_j})^* (\elle)=\\
&=\moltp^n\sum_I\deg \phi_I^*\cbo.\qedhere
\end{align*}
\end{proof}

\section{The Proof of the main theorem}\label{sei}
 Recall that
\[
A=A_1^{N_1}\times\dotsb \times A_r^{N_r},\]
where $A_i$ are pairwise non-isogenous simple abelian varieties. We fixed embeddings $A_i\rightarrow \mathbb{P}_{m_i}$ given by  symmetric ample line bundles on $A_i$. The bundles $\elle$ on $A$ and $\elle_{N_i}$ on $A_{i}^{N_i}$ are obtained as the tensor product of the pullbacks of the natural projections.

We first prove a weak form of Theorem \ref{treapp}, and then we remove the more restrictive hypothesis.
\begin{thm}
\label{canoni}  Theorem \ref{treapp} holds for $H$ an abelian subvariety of $A$. In particular, if $Y$ is a transverse subvariety of $H$, then for any 
$\eta>0$, there exists a positive   constant $c_1(A,\elle,\eta)$ such that 
$$\mu_\mathcal{L} (Y) \ge c_1(A,\elle,\eta)  \frac{
( \deg_\mathcal{L} H)^{\frac{1}{\codim_H Y}- \eta}
}{
(\deg_\mathcal{L} Y)^{\frac{1}{\codim_H Y}+ \eta}
}.$$
\end{thm}
\begin{proof}
Unless specified otherwise, in this proof we keep the same notation as in the previous section. Recall that \[H=H_1\times \ldots \times H_r\] where $H_i$ is an abelian subvariety of $A_i^{N_i}$, set $n_i=\frac{\dim H_i}{\dim A_i}$ and $d=\dim Y$. We set \[\Phi_H=\varphi_{H_1}\times\dotsb\times \varphi_{H_r}:H\to A',\]
 where $A'$ is the abelian variety $A_1^{n_1}\times \dotsb \times A_r^{n_r}$.

Denoting by \[\pi_i:A\to A_i^{N_i}\] the projection on $A_i^{N_i}$, we  have
\[\mathcal{L}=\bigotimes_{i=1}^r \pi_i^* \elle_{N_i}.\]

Recall that \[\mathbb{I}_i=\{(i_1,\ldots,i_{n_i})\in \{1,\ldots,N_i\}^{n_i}\};\] for every \[\Ii=(I_1,\ldots,I_r)\in\mathbb{I}_1\times\dotsb\times\mathbb{I}_r\] let \[\Phi_\Ii=\varphi_{I_1}\times\dotsb\times \varphi_{I_r}:A_1^{n_1}\times \dotsb \times A_r^{n_r}\to A_1^{n_1}\times \dotsb \times A_r^{n_r},\]
which is an isogeny by Lemma \ref{isogenie}.
 
We also define \[\alpha=\max_i \alpha_i^2 \binom{N_i-1}{n_i-1},\]
 where $\alpha_i$ is the minimal positive integer such that $[\alpha_i]=\phi_i\hat\phi_i=\hat\phi_i\phi_i$, with the notations and definitions of Section \ref{equivalencesection}.

Finally denote by $\mathcal{M}$ the bundle on $A_1^{n_1}\times \dotsb \times A_r^{n_r}$ given by \[\mathcal{M}=\bigotimes_{i=1}^r \pi_i^* \elle_{n_i}.\]

   By Lemma \ref{minessmin}, we know that 
   $$\mu_{\bigotimes_\Ii\Phi^*_\Ii \mathcal{M}}(\Phi_{H}({  {Y}})) \ge \sum_\Ii \mu_{\Phi^*_\Ii \mathcal{M}}(\Phi_{H}({  {Y}})).$$

   We apply Theorem \ref{thm1}  to each $\Phi_\Ii$ on $A'$ and $\mathcal{M}$. We deduce that for each $\Ii$, 
   \begin{equation*}
      \mu_{\Phi^*_\Ii \mathcal{M}}(\Phi_{H}({  {Y}})) \gg
      \frac{\left(\deg_ { \Phi^*_\Ii \mathcal{M}}A'\right)^{\frac{1}{\codim_H Y}-\eta}}{\left(\deg_ {\Phi^*_\Ii \mathcal{M}}\Phi_{H}({  {Y}})\right)^{\frac{1}{\codim_H Y}+\eta}}.
   \end{equation*}
   We obtain
   \begin{equation*}\begin{split}
      \mu_{\bigotimes_\Ii\Phi^*_\Ii \mathcal{M}}(\Phi_{H}({  {Y}})) &\geq\sum_\Ii
      \mu_{\Phi^*_\Ii \mathcal{M}}(\Phi_{H}({  {Y}}))\\
      &\gg   
      \sum_\Ii \frac{\left(\deg_ { \Phi^*_\Ii \mathcal{M}}A'\right)^{\frac{1}{\codim_H Y}-\eta}
      }{\left(\deg_ {\Phi^*_\Ii \mathcal{M}}\Phi_{H}({  {Y}})\right)^{\frac{1}{\codim_H Y}+\eta}}.
   \end{split}\end{equation*}
   Since each bundle is ample, for every variety $X$, we have   $\deg_ {\Phi^*_\Ii \mathcal{M}  }X \le \deg_ {\otimes_\Ii\Phi^*_\Ii \mathcal{M}}X $.

   Note also that, for $x_i\ge 1$, $(\sum x_i)^{1/m}\le \sum
    x_i^{1/m}$.

   Recall that, by definition, the degree of the abelian variety is the degree of the line bundle. Using then Proposition \ref{gael}, we deduce
   \begin{multline*}
      \sum_\Ii \frac{\left(\deg_ { \Phi^*_\Ii \mathcal{M}}A'\right)^{\frac{1}{\codim_H Y}-\eta}
      }{\left(\deg_ {\Phi^*_\Ii \mathcal{M}}\Phi_{H}({  {Y}})\right)^{\frac{1}{\codim_H Y}+\eta}}\ge  \frac{\left(\sum_\Ii \deg_ { \Phi^*_\Ii \mathcal{M}}A'\right)^{\frac{1}{\codim_H Y}-\eta}
      }{\left( \deg_ {\otimes_\Ii\Phi^*_\Ii
      \mathcal{M}}\Phi_{H}({  {Y}})\right)^{\frac{1}{\codim_H Y}+\eta}}\geq\\
      \ge {\left(\max_i\binom{N_i-1}{n_i-1}^{n_i}\right)}^{\frac{-1}{\codim_H Y}+\eta} \frac{\left(\deg_ { \bigotimes_\Ii\Phi^*_\Ii \mathcal{M}}A'\right)^{\frac{1}{\codim_H Y}-\eta}
      }{\left(\deg_ {\bigotimes_\Ii\Phi^*_\Ii \mathcal{M}}\Phi_{H}({  {Y}})\right)^{\frac{1}{\codim_H Y}+\eta}}.
   \end{multline*}
   Therefore
   \begin{equation*}\mu_{\bigotimes_\Ii\Phi^*_\Ii \mathcal{M}}(\Phi_{H}({  {Y}})) \gg 
      \frac{\left(\deg_ { \bigotimes_\Ii\Phi^*_\Ii \mathcal{M}}A'\right)^{\frac{1}{\codim_H Y}-\eta}
      }{\left(\deg_ {\bigotimes_\Ii\Phi^*_\Ii \mathcal{M}}\Phi_{H}({  {Y}})\right)^{\frac{1}{\codim_H Y}+\eta}}.
   \end{equation*}

   Note that with our notation $A'=\Phi_{H}({  {H}})$.\\
   Moreover, by  \eqref{push}, $\deg_ { \bigotimes_\Ii\Phi^*_\Ii \mathcal{M}} A' \gg \deg_ { \bigotimes_\Ii\Phi^*_\Ii \mathcal{M}} {\Phi_{H}}_*({  {H}})$.  We deduce 
   \begin{equation}\label{formula1}
      \mu_{\bigotimes_\Ii\Phi^*_\Ii \mathcal{M}}(\Phi_{H}({  {Y}})) \gg 
      \frac{\left(\deg_ { \bigotimes_\Ii\Phi^*_\Ii \mathcal{M}}{\Phi_{H}}_*({  {H}})\right)^{\frac{1}{\codim_H Y}-\eta}
      }{\left(\deg_ {\bigotimes_\Ii\Phi^*_I \mathcal{M}}\Phi_{H}({  {Y}})\right)^{\frac{1}{\codim_H Y}+\eta}}.
   \end{equation}

  By  \eqref{projectionformula} and Theorem \ref{relchiave} we obtain    
\begin{equation}\label{formula2}
      \deg_ {\bigotimes_\Ii \Phi^*_I \mathcal{M}}{\Phi_{H}}_*({  {H}}) =\deg_{\Phi_H^* \bigotimes_\Ii\Phi^*_I \mathcal{M}}{  {H}} \geq \deg_ {\elle}{  {H}}.
      \end{equation}
From Lemma \ref{gradino} and relation \eqref{projectionformula} we have
      \begin{equation}
      \deg_ {\bigotimes_\Ii \Phi^*_\Ii \mathcal{M}}\Phi_{|H}({  {Y}})=\frac{1}{|\ker \Phi_H\cap\stabv|}\deg_{\Phi_H^* \bigotimes_\Ii\Phi^*_I \mathcal{M}}{  {Y}}
\end{equation}
and from Theorem \ref{relchiave} and relation \eqref{gradino1ii} we get
\begin{equation}
\frac{1}{|\ker \Phi_H\cap\stabv|}\deg_{\Phi_H^* \bigotimes_\Ii\Phi^*_I \mathcal{M}}{  {Y}}\leq  \alpha^{\dim Y}\deg_ {\elle}{  {Y}}.
      \end{equation}
Thus
\begin{equation}\label{formula3}
\deg_ {\bigotimes_\Ii \Phi^*_\Ii \mathcal{M}}\Phi_{|H}({  {Y}})\leq \alpha^{\dim Y}\deg_ {\elle}{  {Y}}.
\end{equation}

Finally 
      \begin{equation}\label{formula4}
      \mu_{ \bigotimes_\Ii \Phi^*_\Ii \mathcal{M}}(\Phi_{H}({  {Y}}))=\mu_{\Phi_H^* \bigotimes_\Ii\Phi^*_I \mathcal{M} }({  {Y}})\leq \alpha \mu_\elle ({  {Y}}),
  \end{equation}
where the first equality comes from relation \eqref{phisaltamu}, while the second inequality by Theorem \ref{relchiave} and relation \eqref{gradino1iii}. 

 Notice that, by Lemma \ref{mw},  $\alpha \ll 1$. Plugging inequalities \eqref{formula2}, \eqref{formula3}, \eqref{formula4} into \eqref{formula1}, we deduce

   \[\mu_\elle ({  {Y}})\gg   
   \frac{(\deg_ {\elle}{  {H}})^{\frac{1}{\codim_H Y}-\eta}}{(\deg_ {\elle}{  {Y}})^{\frac{1}{\codim_H Y}+\eta}}.\qedhere\]
   \end{proof}

We can now weaken the hypothesis on $H$ and prove Theorem \ref{treapp} for $Y$ transverse in a translate of an abelian subvariety.

\begin{proof}[Proof of Theorem \ref{treapp}]
 We write the translate $H$ as $H=B+p$ , with $B$ an abelian subvariety.
Let $c_1$ be the constant in Theorem  \ref{canoni}; define 
   $$\theta=c_1  \frac{
   ( \deg_\elle B)^{\frac{1}{\codim_B Y}- \eta}
   }{
   (\deg_\elle Y)^{\frac{1}{\codim_B Y}+ \eta}
   }.$$

   If the set of points of $Y$ of height at most $\frac{1}{4}\theta$ is empty then $\mu_{\elle}(Y)\ge \frac{1}{4}\theta$.
   If not, choose a point $q \in Y$ such that $\hat h_{\elle}(q)\le \frac{1}{4}\theta$.
   We now translate by $-q$, so to have that $Y-q \subseteq B$. Translations preserve transversality and degrees, and so by  Theorem \ref{canoni} for $Y-q  $,
   $$\mu_{\elle}(Y-q  )\ge \theta.$$
   If $x\in Y$ and $\hat h_{\elle}(x)\le \frac{1}{4}\theta$, then $x -q  \in Y-q  $ and 
   $$\hat h_{\elle}(x-q  )\le 2\hat h_{\elle}(x)+2\hat h_{\elle}(q  )\le \theta\le\mu_{\elle}(Y-q  ). $$
   This shows that 
   \[\mu_{\elle}(Y)\ge  \frac{1}{4}\theta.\qedhere\]
\end{proof}

We finally remark that the constant in Theorem \ref{treapp} depends explicitly on the constant in bound \eqref{galkind}. In particular, the constant of Theorem \ref{dueapp} depends explicitly on the one given by Galateau.

\section*{Acknowledgements}
The authors are grateful to the referee for his valuable suggestions and comments.


 \bibliography{Biblio-SharpBogomolov3}

\providecommand{\bysame}{\leavevmode\hbox to3em{\hrulefill}\thinspace}
\providecommand{\MR}{\relax\ifhmode\unskip\space\fi MR }
\providecommand{\MRhref}[2]{%
  \href{http://www.ams.org/mathscinet-getitem?mr=#1}{#2}
}
\providecommand{\href}[2]{#2}
\begin{thebibliography}{DMOyS82}

\bibitem[AD03]{fra}
Francesco Amoroso and Sinnou David, \emph{Minoration de la hauteur
  normalis{\'e}e dans un tore}, J. Inst. Math. Jussieu \textbf{2} (2003),
  no.~3, 335--381. \MR{1990219 (2004m:11101)}

\bibitem[AV09]{iofrancesco}
Francesco Amoroso and Evelina Viada, \emph{Small points on subvarieties of a
  torus}, Duke Math. J. \textbf{150} (2009), no.~3, 407--442. \MR{2582101
  (2010m:11080)}

\bibitem[Ber87]{bertrand}
Daniel Bertrand, \emph{Minimal heights and polarizations on abelian varieties},
  MSRI, preprint 06220-87, June 1987.

\bibitem[BL04]{l-b}
Christina Birkenhake and Herbert Lange, \emph{Complex abelian varieties},
  second ed., Grundlehren der Mathematischen Wissenschaften [Fundamental
  Principles of Mathematical Sciences], vol. 302, Springer-Verlag, Berlin,
  2004. \MR{2062673 (2005c:14001)}

\bibitem[BS04]{BakSil}
Matthew Baker and Joseph Silverman, \emph{A lower bound for the canonical
  height on abelian varieties over abelian extensions}, Math. Res. Lett.
  \textbf{11} (2004), no.~2-3, 377--396. \MR{2067482 (2005e:11083)}

\bibitem[BZ96]{BZ}
Enrico Bombieri and Umberto Zannier, \emph{Heights of algebraic points on
  subvarieties of abelian varieties}, Ann. Scuola Norm. Sup. Pisa Cl. Sci. (4)
  \textbf{23} (1996), no.~4, 779--792 (1997). \MR{1469574 (98j:11043)}

\bibitem[Car09]{carrizosaIMRN}
Mar{\'\i}a Carrizosa, \emph{Petits points et multiplication complexe}, Int.
  Math. Res. Not. IMRN (2009), no.~16, 3016--3097. \MR{2533796 (2011c:11102)}

\bibitem[CVV]{TAI}
Sara Checcoli, Francesco Veneziano, and Evelina Viada, \emph{On torsion
  anomalous intersections (with an appendix by {P}. {P}hilippon)}, Preprint
  2012.

\bibitem[DH00]{davidhindry}
Sinnou David and Marc Hindry, \emph{Minoration de la hauteur de
  {N}{\'e}ron-{T}ate sur les vari{\'e}t{\'e}s ab{\'e}liennes de type {C}. {M}},
  J. Reine Angew. Math. \textbf{529} (2000), 1--74. \MR{1799933 (2001j:11054)}

\bibitem[DMOyS82]{OgusDensitySurf}
Pierre Deligne, James Milne, Arthur Ogus, and Kuang yen Shih, \emph{Hodge
  cycles, motives, and {S}himura varieties}, Lecture Notes in Mathematics, vol.
  900, Springer-Verlag, Berlin, 1982. \MR{654325 (84m:14046)}

\bibitem[DP98]{DavidPhilipponTiruchirapalli}
Sinnou David and Patrice Philippon, \emph{Minorations des hauteurs
  normalis{\'e}es des sous-vari{\'e}t{\'e}s de vari{\'e}t{\'e}s
  ab{\'e}liennes}, Number theory ({T}iruchirapalli, 1996), Contemp. Math., vol.
  210, Amer. Math. Soc., Providence, RI, 1998, pp.~333--364. \MR{1478502
  (98j:11044)}

\bibitem[DP02]{sipacommentari}
\bysame, \emph{Minorations des hauteurs normalis{\'e}es des
  sous-vari{\'e}t{\'e}s de vari{\'e}t{\'e}s abeliennes. {II}}, Comment. Math.
  Helv. \textbf{77} (2002), no.~4, 639--700. \MR{1949109 (2004a:11055)}

\bibitem[DP07]{sipaIMRP}
\bysame, \emph{Minorations des hauteurs normalis{\'e}es des
  sous-vari{\'e}t{\'e}s des puissances des courbes elliptiques}, Int. Math.
  Res. Pap. IMRP (2007), no.~3, Art. ID rpm006, 113. \MR{2355454 (2008h:11068)}

\bibitem[Gal10]{galateau}
Aur{\'e}lien Galateau, \emph{Une minoration du minimum essentiel sur les
  vari{\'e}t{\'e}s ab{\'e}liennes}, Comment. Math. Helv. \textbf{85} (2010),
  no.~4, 775--812. \MR{2718139 (2011i:11110)}

\bibitem[Hin88]{Hin}
Marc Hindry, \emph{Autour d'une conjecture de {S}erge {L}ang}, Invent. Math.
  \textbf{94} (1988), no.~3, 575--603. \MR{969244 (89k:11046)}

\bibitem[Mas87]{MasserSmallValues87}
David Masser, \emph{Small values of heights on families of abelian varieties},
  Diophantine approximation and transcendence theory ({B}onn, 1985), Lecture
  Notes in Math., vol. 1290, Springer, Berlin, 1987, pp.~109--148. \MR{927559
  (89g:11048)}

\bibitem[Mum70]{MumfAV}
David Mumford, \emph{Abelian varieties}, Tata Institute of Fundamental Research
  Studies in Mathematics, No. 5, Published for the Tata Institute of
  Fundamental Research, Bombay, 1970. \MR{0282985 (44 \#219)}

\bibitem[MW85]{MWZeroEst}
David Masser and Gisbert W{\"u}stholz, \emph{Zero estimates on group varieties.
  {II}}, Invent. Math. \textbf{80} (1985), no.~2, 233--267. \MR{788409
  (86h:11054)}

\bibitem[MW93]{Masserwustholz}
\bysame, \emph{Periods and minimal abelian subvarieties}, Ann. of Math. (2)
  \textbf{137} (1993), no.~2, 407--458. \MR{1207211 (94g:11040)}

\bibitem[Pin98]{PinkConjDensity}
Richard Pink, \emph{{$l$}-adic algebraic monodromy groups, cocharacters, and
  the {M}umford-{T}ate conjecture}, J. Reine Angew. Math. \textbf{495} (1998),
  187--237. \MR{1603865 (98m:11060)}

\bibitem[PS08]{SombraPatrice}
Patrice Philippon and Mart{\'\i}n Sombra, \emph{Minimum essentiel et degr{\'e}s
  d'obstruction des translat{\'e}s de sous-tores}, Acta Arith. \textbf{133}
  (2008), no.~1, 1--24. \MR{2413362 (2009g:11079)}

\bibitem[Ray83a]{RaynaudMM}
Michel Raynaud, \emph{Courbes sur une vari{\'e}t{\'e} ab{\'e}lienne et points
  de torsion}, Invent. Math. \textbf{71} (1983), no.~1, 207--233. \MR{688265
  (84c:14021)}

\bibitem[Ray83b]{RaynaudMMgen}
\bysame, \emph{Sous-vari{\'e}t{\'e}s d'une vari{\'e}t{\'e} ab{\'e}lienne et
  points de torsion}, Arithmetic and geometry, {V}ol. {I}, Progr. Math.,
  vol.~35, Birkh{\"a}user Boston, Boston, MA, 1983, pp.~327--352. \MR{717600
  (85k:14022)}

\bibitem[Sch96]{SchmidtHeightsPoints}
Wolfgang Schmidt, \emph{Heights of points on subvarieties of {${\bf G}^n\_m$}},
  Number theory ({P}aris, 1993--1994), London Math. Soc. Lecture Note Ser.,
  vol. 235, Cambridge Univ. Press, Cambridge, 1996, pp.~157--187. \MR{1628798
  (99h:11070)}

\bibitem[Ser89]{SerreDensityEllCurv}
Jean-Pierre Serre, \emph{Abelian {$l$}-adic representations and elliptic
  curves}, second ed., Advanced Book Classics, Addison-Wesley Publishing
  Company Advanced Book Program, Redwood City, CA, 1989, With the collaboration
  of Willem Kuyk and John Labute. \MR{1043865 (91b:11071)}

\bibitem[Sil84]{SilvermanLowerBoundsDuke}
Joseph Silverman, \emph{Lower bounds for height functions}, Duke Math. J.
  \textbf{51} (1984), no.~2, 395--403. \MR{747871 (87d:11039)}

\bibitem[Ull98]{UllmoBogo}
Emmanuel Ullmo, \emph{Positivit{\'e} et discr{\'e}tion des points
  alg{\'e}briques des courbes}, Ann. of Math. (2) \textbf{147} (1998), no.~1,
  167--179. \MR{1609514 (99e:14031)}

\bibitem[Via09]{ioirmn}
Evelina Viada, \emph{Nondense subsets of varieties in a power of an elliptic
  curve}, Int. Math. Res. Not. IMRN (2009), no.~7, 1213--1246. \MR{2495303
  (2010d:14031)}

\bibitem[Via10]{ijnt}
\bysame, \emph{Lower bounds for the normalized height and non-dense subsets of
  subvarieties of abelian varieties}, Int. J. Number Theory \textbf{6} (2010),
  no.~3, 471--499. \MR{2652892 (2011i:11111)}

\bibitem[Zha98]{ZhangEquidistribution}
Shou-Wu Zhang, \emph{Equidistribution of small points on abelian varieties},
  Ann. of Math. (2) \textbf{147} (1998), no.~1, 159--165. \MR{1609518
  (99e:14032)}

\end{thebibliography}
 \bibliographystyle{amsalpha} 
\end{document}